\documentclass[11pt]{amsart}
\usepackage{amsfonts}
\usepackage{amsmath}
\usepackage{amssymb}
\usepackage[T1]{fontenc}
\usepackage{textcomp}

\newtheorem{theorem}{Theorem}[section]

\newtheorem{lemma}[theorem]{Lemma}

\newtheorem{remark}[theorem]{Remark}

\numberwithin{equation}{section}

\begin{document}
\title{A Gap Theorem on complete shrinking gradient Ricci solitons}

\author{Shijin Zhang}
\address{Shijin Zhang, School of Mathematics and systems science, Beihang University, Beijing,  100871, P.R.China, shijinzhang@buaa.edu.cn}

 \subjclass[2000] {Primary: 53C20}
\keywords {shrinking gradient Ricci solitons, sectional curvature, gap theorem}

\maketitle

\begin{abstract}
In this short note, using G\"unther's volume comparison theorem and Yokota's gap theorem on complete shrinking gradient Ricci solitons, we prove that for any complete shrinking gradient Ricci soliton $(M^{n},g,f)$ with sectional curvature $K(g)<A$ and ${\rm Vol}_{f}(M)\geq v$ for some uniform constant $A,v$,  there exists a small uniform constant $\epsilon_{n,A,v}>0$ depends only on $n, A$ and $v$, if the scalar curvature $R\leq \epsilon_{n,A,v}$,  then $(M,g,f)$ is isometric to the Gaussian soliton $(\mathbb{R}^{n}, g_{E}, \frac{|x|^{2}}{4})$.
\end{abstract}

\vskip6mm
\section*{Introduction}
Gradient Ricci solitons play an important role in Hamilton's Ricci flow as they correspond to self-similar solutions, and often arise as singularity models of the Ricci flow.  The study of Ricci solitons has also become increasingly important in the study in metric measure theory.

A complete Riemannian manifold $(M,g)$ is called a gradient Ricci soliton if there exists a smooth function $f$ on $M$ such that
\begin{equation*}
{\rm Ric}+\nabla\nabla f=\lambda g
\end{equation*}
for some constant $\lambda$. It is denoted by $(M,g,f)$. For $\lambda<0$ the Ricci soliton is expanding, for $\lambda=0$ it is steady and for $\lambda >0$ is shrinking. The function $f$ is called a potential function of the gradient Ricci soliton. After rescaling the metric $g$ we may assume that $\lambda\in\{-\frac{1}{2},0,\frac{1}{2}\}$. In this short note, we only consider the shrinking case. If $M=\mathbb{R}^{n}, g=g_{E}$ (Euclidean metric) , $f=\frac{|x|^{2}}{4}$, $(M,g,f)$ is a shrinking Ricci soliton, it is called a Gaussian soliton.

In this note, $R$ denotes the scalar curvature of $(M,g,f)$, we always normalize the potential function $f$ by adding a constant so that
\begin{equation}
R+|\nabla f|^{2}=f.
\end{equation}
With this normalization of $f$, the normalized $f$-volume ${\rm Vol}_f(M)$, which is called the Gaussian density in \cite{CHI},  is defined by 
\begin{equation}
{\rm Vol}_{f}(M)=(4\pi)^{-\frac{n}{2}}\int_{M}e^{-f}dV_{g},
\end{equation}
and Perelman's invariant $\mu_0$ is defined by \cite{Carrillo-Ni, Perelman}:
\begin{equation}
\mu_0=-\log{\rm Vol}_{f}(M).
\end{equation}

For complete shrinking gradient Ricci soliton $(M,g,f)$, if $0\leq {\rm Ric}(g)<\frac{1}{2}g$, Naber \cite{Naber} proved that $(M,g,f)$ is isometric to the Gaussian soliton. Naber used the equation of the scalar curvature
\begin{equation}
\Delta_{f}R=R-2|{\rm Ric}|^{2}=\lambda_{i}(1-2\lambda_{i})\geq 0,
\end{equation}
where $\Delta_{f}R=\Delta R-<\nabla f,\nabla R>$, $\lambda_{i}$ are the eigenvalues of Ricci tensor. Then the scalar curvature must be constant (this result also can be obtained from the main result in the author and Ge's paper \cite{Ge-Zhang}), implies that the scalar curvature and the Ricci curvature must be zero, hence $(M,g, f)$ is the Gaussian soliton.

Without the assumption of nonnegative Ricci curvature, Munteanu and Wang \cite{Munteanu-WangM} proved that if $|{\rm Ric}|\leq \frac{1}{100n}$, then $(M,g,f)$ is isometric to the Gaussian soliton. If the Ricci curvature is bounded, they first proved that the Riemann curvature tensor grows at most polynomially in the distance function, then using the equations for the scalar curvature, Ricci tensor and Riemann curvature tensor on the complete shrinking gradient Ricci solitons, they proved that for any $p\geq 3$,
\begin{equation}
\int_{M}(f-\frac{n}{2}+p(1-50pK))|{\rm Rm}|^{p}e^{-f}dV_{g}\leq 0,
\end{equation}
where $K=\sup_{x\in M}|{\rm Ric}|$. Then they take $p=n$ and $K=\frac{1}{100n}$, to get ${\rm Rm}=0$, hence $(M,g,f)$ is isometric to the Gaussion soliton.

In this short note, using G\"unther's volume comparison theorem and Yotoka's gap theorem on the complete shrinking gradient Ricci solitons, we prove the following gap theorem.
\begin{theorem}\label{MainThm}
Let $A,v$ be positive numbers. Let $(M^{n},g,f)$ be a complete shrinking gradient Ricci soliton with sectional curvature $K(g)<A$ and ${\rm Vol}_{f}(M)\geq v$. There exists $\epsilon_{n, A,v}>0$ depends only on $n, A$ and $v$, if $R(g)\leq \epsilon_{n, A,v}$, then $(M,g,f)$ is isometric to  the Gaussian soliton $(\mathbb{R}^{n}, g_{E}, \frac{|x|^{2}}{4})$.
\end{theorem}

We recall Yokota's gap theorem as following.
\begin{theorem}[Yokota, \cite{Yokota,YokotaA}]\label{Yokota}
There exists a constant $\epsilon'_{n}>0$ which depends only on $n\geq 2$ and satisfies the following: Any complete shrinking gradient Ricci soliton $(M^{n},g,f)$ with
$${\rm Vol}_{f}(M)\geq 1-\epsilon'_{n},$$
then $(M,g,f)$ is, up to scaling, the Gaussian soliton $(\mathbb{R}^{n}, g_{E}, \frac{|x|^{2}}{4})$.
\end{theorem}

In \cite{Yokota}, using Perelman's reduced volume, Yokota proved a gap theorem for ancient solutions to the Ricci flow with Ricci curvature bounded below. As a corollary, he obtained the above gap theorem under the additional assumption that the Ricci curvature is bounded below. Later, in \cite{YokotaA}, he removed the assumption that the Ricci curvature is bounded below. As an application of the above theorem, Yokota gave a complete affirmative answer to the conjecture of Carrillo-Ni \cite{Carrillo-Ni}, the complete shrinking gradient Ricci soliton $(M,g,f)$ is the Gaussian soliton if and only if ${\rm Vol}_{f}(M)=1.$

Now we recall G\"unther's volume comparison theorem. Let $M_{H}$ be the simply connected space form of constant sectional curvature $H$. Let $V_{H}(r)$ denote the volume of a ball in $M_{H}$ with radius $r$. The G\"unther's volume comparison theorem is following. 

\begin{theorem}[G\"unther,\cite{Gunther}] \label{Gunther}
Let $(M,g)$ be a Riemannian manifold, $p\in M$, $i(p)$ denotes the injectivity radius at the point $p$. If the sectional curvature $K(g)\leq H$, then
\begin{equation}
{\rm Vol}(B_p(r))\geq V_{H}(r)
\end{equation}
for any $r<\min\{i(p),\frac{\pi}{H^{1/2}}\}$.
\end{theorem}
Note that for $H>0$
$$V_{H}(r)=n\omega_{n}\int_{0}^{r}(\frac{\sin(H^{1/2}s)}{H^{1/2}})^{n-1}ds,$$
here $\omega_{n}$ is the volume of unit ball $B(1)$ in $\mathbb{R}^{n}$ with Euclidean metric $g_{E}$.

Under the assumption of ${\rm Vol}_{f}(M)\geq v$, we use Carrillo-Ni's logarithmic Sobolev inequality to get the lower bound of ${\rm Vol}(B_{p}(1))$, then get the uniform lower bound of the injectivity radius from Cheeger-Gromov-Taylor's theorem (Theorem \ref{Cheeger-Gromov-Taylor} below),  and use the G\"unther's volume comparison theorem,  we get a better lower bound of ${\rm Vol}(B_{p}(r_0))$ with small $r_0$. Thus we get a lower bound of the normalized $f$-volume with small scalar curvature.  Then the Theorem \ref{MainThm} follows from Yokota's gap theorem (Theorem \ref{Yokota}). 

In section 1, under the assumption that the scalar curvature $R\leq \delta<\frac{n}{2}$, we give the lower bound of the volume of $D(r)$, the set of $x\in M$ such that $2\sqrt{f(x)-f(p)}\leq r$. In section 2, we obtain the uniform lower bound of the injectivity radius. In section 3, we give the proof of the Theorem \ref{MainThm}.

\section{Lower bound of the volume of $D(r)$}
In this section, recall some properties of the complete shrinking gradient Ricci solitons and give a estimate about the lower bound of the volume of $D(r)$ with scalar curvature $R\leq \delta<\frac{n}{2}$.

\begin{lemma}
\quad
\begin{enumerate}
\item $R+\Delta f=\frac{n}{2}$;
\item $R+|\nabla f|^{2}=f$, after normalizing the function $f$ by a constant;
\item For any fixed point $p\in M^{n}$, there exist two positive constants $c_{1}$ and $c_{2}$ so that, for any $x\in M^{n}$, we have
\begin{equation}
\frac{1}{4}(d(x)-c_{1})^{2}\leq f(x)\leq \frac{1}{4}(d(x)+c_{2})^{2},
\end{equation}
where $d(x)$ is the distance function from $x$ to $p$.
\end{enumerate}
\end{lemma}
(3) in the above lemma is proved by Cao-Zhou \cite{Cao-Zhou}(see also Fang-Man-Zhang \cite{Fang-Man-Zhang} and for an improvement, Haslhofer-M\"{u}ller \cite{Haslhofer-Muller}). It is well known that a complete shrinking gradient Ricci solitons has nonnegative scalar curvature (see Chen \cite{Chen}) and either $R>0$ or the metric $g$ is flat (see Pigola-Rimoldi-Setti \cite{Pigola-Rimoldi-Setti2} or the author \cite{Zhang}). Recently, Chow-Lu-Yang \cite{Chow-Lu-Yang} proved that the scalar curvature of a complete noncompact nonflat shrinker has a lower bound by $Cd^{-2}(x)$ for some positive constant $C$. From (3) in the above lemma, we know there exists a point $p\in M$ such that $\nabla f(p)=0$, we will fix this point in this note.

We denote $c_{p}=f(p)=R(p)\geq 0$, $\rho=2 \sqrt{f-c_{p}}$, and $D(r)=\{\rho\leq r \}$.
Define $V(r)=\int_{D(r)}dv_{g}$, and $\chi(r)=\int_{D(r)}R dv_{g}$, then
$V'(r)=\int_{\partial D(r)}\frac{1}{|\nabla \rho|}=\frac{r}{2}\int_{\partial D(r)}\frac{1}{|\nabla f|}$
and $\chi'(r)=\int_{\partial D(r)}\frac{R}{|\nabla \rho|}=\frac{r}{2}\int_{\partial D(r)}\frac{R}{|\nabla f|}$.

\begin{lemma}\label{volume and scalar curvature integral identity}
\begin{equation}
0\leq n V(r)-2\chi(r)=\frac{r^{2}+4c_{p}}{r} V'(r)-\frac{4}{r}\chi'(r).
\end{equation}
\end{lemma}
\begin{proof}
Since
$$\Delta f+R=\frac{n}{2},$$
we have
\begin{eqnarray*}
\begin{aligned}
n V(r)-2\int_{D(r)}R dv_{g} &= 2\int_{D(r)}\Delta f dv_{g}\\
&= 2\int_{\partial D(r)}|\nabla f|\\
&=2\int_{\partial D(r)}\frac{f-R}{|\nabla f|}\\
&= \frac{r^{2}+4c_{p}}{r} V'(r)-2\int_{\partial D(r)}\frac{R}{|\nabla f|}\\
&= \frac{r^{2}+4c_{p}}{r} V'(r)-\frac{4}{r}\chi'(r).
\end{aligned}
\end{eqnarray*}
\end{proof}
Then we can estimate the volume growth from below under $R\leq \delta<\frac{n}{2}$ for some positive $\delta$.
\begin{lemma}\label{Lemma1.3}
Let $(M,g,f)$ be a complete shrinking gradient Ricci soliton with scalar curvature $R\leq \delta<\frac{n}{2}$ for some positive $\delta$. Then for any $r_{0}> 0$ and for any $r>r_{0}$, we have
\begin{equation}\label{VolumeLowerBound}
V(r)\geq \frac{V(r_{0})}{(r^{2}_{0}+4c_{p})^{\frac{n}{2}-\delta}}(r^{2}+4c_{p})^{\frac{n}{2}-\delta},
\end{equation}
and
\begin{equation}\label{AreaLowerBound}
V'(r)\geq (n-2\delta)\frac{V(r_{0})}{(r^{2}_{0}+4c_{p})^{n/2-\delta}}r(r^{2}+4c_{p})^{\frac{n}{2}-1-\delta}.
\end{equation}
\end{lemma}
\begin{proof}
\begin{eqnarray*}
\begin{aligned}
\frac{d}{dr}(\log \frac{V(r)}{(r^{2}+4c_{p})^{\frac{n}{2}-\delta}})
&=\frac{V'(r)}{V(r)}-(\frac{n}{2}-\delta)\frac{2r}{r^{2}+4c_{p}}\\
&=\frac{(r^{2}+4c_{p})V'(r)-(n-2\delta)rV(r)}{(r^{2}+4c_{p})V(r)}\\
&=\frac{2r(\delta V(r)-\chi(r))+4\chi'(r)}{(r^{2}+4c_{p})V(r)}
\end{aligned}
\end{eqnarray*}
where we have used Lemma \ref{volume and scalar curvature integral identity} in the last equality. Since $R\leq \delta$ and $\chi(r)$ is nondecreasing with $r$, we obtain that
\begin{equation}
\frac{d}{dr}(\log \frac{V(r)}{(r^{2}+4c_{p})^{\frac{n}{2}-\delta}})\geq 0.
\end{equation}
Then we get (\ref{VolumeLowerBound}). Since $\chi(r)\leq \delta V(r)$ and $\chi'(r)\geq 0$, then (\ref{AreaLowerBound}) follows from Lemma \ref{volume and scalar curvature integral identity} and (\ref{VolumeLowerBound}).
\end{proof}
\begin{remark}
Chen Chih-Wei \cite{ChenCW} also obtain the lower bound of the geodesic ball with $\frac{1}{{\rm Vol}(B_r)}\int_{B_r}Rdv_g\leq \delta<\frac{n}{2}.$ Cao-Zhou \cite{Cao-Zhou} proved that the upper bound ${\rm Vol}(B_r)\leq Cr^{n}$ for any complete shrinking gradient Ricci solitons, and Chow-Lu-Yang \cite{Chow-Lu-Yang1} concluded a criterion for a shrinking gradient Ricci soliton to have positive asymptotic volume ratio.
\end{remark}

\section{Lower bound of the injectivity radius}
In this section, we prove the lower bound of the injectivity radius. We recall the logarithmic Sobolev inequality for shrinking Ricci solitons established by Carrillo-Ni \cite{Carrillo-Ni}.
\begin{equation}\label{LogSobolevInequality}
\int_{M}u^{2}\log u^{2}dv-(\int_{M}u^{2}dv)\log(\int_{M}u^{2}dv)\leq \mu_0\int_{M}u^{2}dv+\int_{M}Ru^{2}dv+4\int_{M}|\nabla u|^{2}dv.
\end{equation}
for any $u\in C_{0}^{\infty}(M)$, where $\mu_0$ is a Perelman's invariant. 

Using this logarithmic Sobolev inequality, following Perelman's proof of his no local collapsing theorem (see \cite{Perelman, Kleiner-Lott, Topping}),  we can obtain the following no local collapsing theorem for shrinking Ricci solitons, see \cite{Carrillo-Ni}.
\begin{theorem}
Let $(M,g,f)$ be a complete shrinking gradient Ricci soliton, then there exists a constant $\kappa=\kappa(\mu_0,n)$ depends only on $\mu_0$ and $n$ satisfying the following property. If $x_0\in M$ is a point and $r_0>0$ are such that $R\leq r_{0}^{-2}$ in $B_{x_0}(r_0)$, then 
\begin{equation}
{\rm Vol}B_{x_0}(r_0)\geq \kappa r_0^{n}.
\end{equation}
\end{theorem}
Here we just need a special case of no local collapsing theorem under the assumptions in the Theorem \ref{MainThm}, we state it as the following lemma.
\begin{lemma}
Let $(M,g,f)$ be a complete shrinking gradient Ricci soliton with ${\rm Vol}_{f}(M)\geq v>0$, $K(g)\leq A$ and $R\leq 1$, then there exists a uniform constant $C(A,n)>0$ depends only on $A$ and $n$ such that for any point $p\in M$
\begin{equation}
{\rm Vol}(B_{p}(1))\geq C(A,n)v.
\end{equation}
\end{lemma}
\begin{proof}
For short we denote $B(r)$ for $B_{p}(r)$ and $V(r)$ for ${\rm Vol}(B_{p}(r))$ in the proof of this lemma.

Let $\phi:[0,\infty)\rightarrow [0,1]$ be a cutoff function with $\phi=1$ as $x\in[0,1/2]$, $\phi=0$ as $x\in [1,\infty)$ and $|\phi'|\leq 3$. Set 
\begin{equation*}
\omega^{2}_{1}(x)=(4\pi)^{-\frac{n}{2}}\phi^{2}(r(x))e^{-C}
\end{equation*}
where $C$ is chosen so that $\int_{M}\omega^{2}_1dv=1.$ We have the estimate of $C$,
\begin{equation}
-\frac{n}{2}\log 4\pi +\log V(\frac{1}{2})\leq C\leq -\frac{n}{2}\log 4\pi +\log V(1).
\end{equation}
Now we estimate
\begin{eqnarray}
\begin{aligned}
4\int_{M}|\nabla w_1|^{2}dv&=4(4\pi)^{-n/2}e^{-C}\int_{M}|\phi'|^{2}dv\\
&\leq 36 (4\pi)^{-n/2}e^{-C}V(1)\\
&\leq 36\frac{V(1)}{V(\frac{1}{2})}.
\end{aligned}
\end{eqnarray}
Since $w^{2}_1=(4\pi)^{-\frac{n}{2}}e^{-C}\phi^{2}$,
$$\int_{M}w^{2}_1\log w^{2}_1 dv=-C-\frac{n}{2}\log(4\pi)+(4\pi)^{-\frac{n}{2}}e^{-C}\int_{M}\phi^{2}\log\phi^{2}dv.$$
Since $x\log x\geq -\frac{1}{e}$ for any $x>0$ and $\phi^{2}\log \phi^{2}=0$ outside the ball $B(1)$, 
$$\int_{M}\phi^{2}\log\phi^{2}dv\geq -\frac{1}{e}V(1).$$
Then using the estimate of $C$, we get
\begin{equation}
\int_{M}w^{2}_1\log w^{2}_1 dv\geq -\log V(1)-\frac{1}{e}\frac{V(1)}{V(1/2)}.
\end{equation}
Since $R\leq 1$, we take $u=w_1$ in the logarithmic Sobolev inequality (\ref{LogSobolevInequality}), and combine the above estimates, we get
$$-\log V(1)-\frac{1}{e}\frac{V(1)}{V(1/2)}\leq \mu_0+1+36\frac{V(1)}{V(1/2)}.$$
Hence we have
$$\log V(1)\geq-\mu_0-1-37\frac{V(1)}{V(1/2)}.$$
Since $K(g)\leq A$ and $R\geq 0$, there exists a uniform constant $C_1=C_{1}(A,n)>0$ depends only on $A$ and $n$ such that $K(g)\geq -C_1$. Then using the Bishop-Gromov's volume comparison theorem, there exists a uniform constant $C_2=C_{2}(A,n)>0$ depends only on $A$ and $n$ such that 
$$\frac{V(1)}{V(1/2)}\leq C_2.$$ 
Hence there exists a uniform constant $C(A,n)>0$ depends only on $A$ and $n$, such that
\begin{equation}
V(1)\geq C(A,n){\rm Vol}_{f}(M)\geq C(A,n)v.
\end{equation}
\end{proof}

Cheeger-Gromov-Taylor obtained a lower bound estimate of the injectivity radius as following, see Theorem 4.7 (i) in \cite{Cheeger-Gromov-Taylor}.
\begin{theorem}[Cheeger-Gromov-Taylor]\label{Cheeger-Gromov-Taylor}
Let $(M^{n},g)$ be a complete manifold with $H\leq K(g)\leq K, K>0.$ Let $r=r(x)$ be the distance function from $x$ to $p$, and fix $r_{1}, r_0, s,$ with $s\leq r_1, r_{0}+2s<\frac{\pi}{\sqrt{K}}, r_{0}\leq \frac{\pi}{4\sqrt{K}}.$ Then
\begin{equation}
i(x)\geq \frac{r_0}{2}\frac{1}{1+(V_{H}(r_0+s)/{\rm Vol}(B_p(r_1)))(V_{H}(r+r_1)/V_{H}(s))}.
\end{equation}
\end{theorem}
Let $i(x)$ denote the injectivity radius at point $x$, let $i(M)$ denote the injectivity of $M$, i.e., $i(M)=\inf_{x\in M} i(x).$Using Theorem \ref{Cheeger-Gromov-Taylor}, we can obtain the uniform lower bound of the injectivity radius.
\begin{lemma}\label{LowerBoundInjectivity}
Let $(M,g,f)$ be a complete shrinking gradient Ricci soliton with the same assumptions in the Theorem \ref{MainThm}, then there exists a uniform constant $C=C(A,v,n)>0$ depends only on $A,v$ and $n$ such that
\begin{equation}
i(M)\geq C.
\end{equation}
\end{lemma}
\begin{proof}
For any point $p\in M$. Under the assumptions in the Theorem \ref{MainThm}, we get a lower bound $H$ of $K(g)$ depends only on $A,n$,  we take $r_{1}=r_{0}=s=\frac{\pi}{4\sqrt{A}}$ and $x=p$ in the Theorem \ref{Cheeger-Gromov-Taylor}, we obtain 
$$i(p)\geq \frac{r_0}{2}\frac{1}{1+V_{H}(2r_0)/{\rm Vol}(B_p(r_0))}.$$
If $r_{0}\geq 1$, then ${\rm Vol}(B_p(r_0))\geq {\rm Vol}(B_p(1))\geq C(A,n)v$. If $r_{0}<1$, by Bishop-Gromov volume comparison theorem, 
$${\rm Vol}(B_p(r_0))\geq \frac{V_{H}(r_0)}{V_{H}(1)}{\rm Vol}(B_p(1))\geq \tilde{C}(A,n)v.$$
Hence there exists a uniform constant $C=C(A,n,v)>0$ depends only on $A,v$, and $n$, such that
$$i(p)\geq C.$$
Hence
$$i(M)\geq C.$$
\end{proof}

\section{Proof of the Theorem \ref{MainThm}}
In this section, we first use the G\"unther's volume comparison theorem to give a lower bound estimate of $V(r_{0})$ for small $r_{0}$ with small scalar curvature. Then by Lemma \ref{Lemma1.3}, we get the estimate ${\rm Vol}_{f}(M)$ is close to $1$, hence the Theorem \ref{MainThm} follows from the Theorem \ref{Yokota}.
\begin{proof}[Proof of Theorem \ref{MainThm}]
We assume the scalar curvature $R\leq \epsilon$ for some $0<\epsilon\leq 1$. By Lemma \ref{LowerBoundInjectivity} and G\"unther's volume comparison theorem, for any $r<C_{0}=\min \{C, \frac{\pi}{\sqrt{A}}\}$, here $C$ is the constant in the Lemma \ref{LowerBoundInjectivity}, we have
\begin{equation}
{\rm Vol}(B_p(r))\geq n\omega_{n}\int_{0}^{r}(\frac{\sin(A^{1/2}s)}{A^{1/2}})^{n-1}ds.
\end{equation}
From $R+|\nabla f|^{2}=f$ and $R\geq 0$, we have $|\nabla 2\sqrt{f}|\leq 1$, hence
$$\rho(x)\leq 2\sqrt{f(x)}\leq r(x)+2\sqrt{f(p)}\leq r(x)+2\sqrt{\epsilon}.$$
Thus  for any $r_{0}\geq 2\sqrt{\epsilon}$ we have
$$B_{p}(r_0-2\sqrt{\epsilon})\subseteq D(r_0).$$
If $r_0-2\sqrt{\epsilon}<C_0$, by G\"unther's volume comparison theorem
\begin{equation}
V(r_0)\geq {\rm Vol}(B_{p}(r_0-2\sqrt{\epsilon})\geq n\omega_{n}\int_{0}^{r_0-2\sqrt{\epsilon}}(\frac{\sin(A^{1/2}s)}{A^{1/2}})^{n-1}ds
\end{equation}
Now we choose $\epsilon=\epsilon_n(A,v)>0$ such that $\epsilon^{1/4}<C_0$ and take $r_{0}=\epsilon^{1/4}+2\epsilon^{1/2}$, then we have
\begin{equation}
V(\epsilon^{1/4}+2\epsilon^{1/2})\geq n\omega_{n}\int_{0}^{\epsilon^{1/4}}(\frac{\sin(A^{1/2}s)}{A^{1/2}})^{n-1}ds
\end{equation}
Hence
\begin{eqnarray}
\begin{aligned}
\frac{V(r_0)}{(r^{2}_0+4c_p)^{n/2-\epsilon}}&\geq\frac{V(\epsilon^{1/4}+2\epsilon^{1/2})}{((\epsilon^{1/4}+2\epsilon^{1/2})^{2}+4\epsilon)^{n/2-\epsilon}}\\
&\geq n\omega_{n}\frac{\int_{0}^{\epsilon^{1/4}}(\frac{\sin(A^{1/2}s)}{A^{1/2}})^{n-1}ds}{\epsilon^{n/4}(1+4\epsilon^{1/4}+8\epsilon^{1/2})^{n/2-\epsilon}\epsilon^{-\epsilon/2}}
\end{aligned}
\end{eqnarray}
We denote
$$\alpha(\epsilon):=(1+4\epsilon^{1/4}+8\epsilon^{1/2})^{-n/2+\epsilon}\epsilon^{\epsilon/2}$$
and
$$C(A, \epsilon):=n\omega_{n}\int_{0}^{\epsilon^{1/4}}(\frac{\sin(A^{1/2}s)}{A^{1/2}})^{n-1}ds/\epsilon^{n/4}.$$
Note that
\begin{equation}
\lim_{\epsilon\rightarrow 0}\alpha(\epsilon)=1,
\end{equation}
and
\begin{equation}
\lim_{\epsilon\rightarrow 0}C(A, \epsilon)=\omega_{n}.
\end{equation}
By Lemma \ref{Lemma1.3},
for any $r\geq \epsilon^{1/4}+2\epsilon^{1/2}$,  we have
\begin{equation}\label{1.21}
V'(r)\geq (n-2\epsilon)\alpha(\epsilon)C(A,\epsilon)r^{n-1-2n\epsilon}.
\end{equation}
Denote $V_{f}(r):=(4\pi)^{-n/2}\int_{D(r)}e^{-f}dv_{g}$. Then 
$$V_{f}(r)=e^{-c_{p}}(4\pi)^{-n/2}\int_{D(r)}e^{-\frac{\rho^{2}}{4}}dv_{g}$$ 
and 
\begin{equation}\label{1.22}
V'_{f}(r)=e^{-c_{p}-\frac{r^{2}}{4}}(4\pi)^{-n/2}\int_{\rho(x)=r}\frac{1}{|\nabla\rho|}=e^{-c_{p}-\frac{r^{2}}{4}}(4\pi)^{-n/2}V'(r).
\end{equation}
Hence for any $r>r_0=\epsilon^{1/4}+2\epsilon^{1/2}$, by (\ref{1.21}) and (\ref{1.22}), we have
\begin{eqnarray*}
\begin{aligned}
&V_{f}(r)=V_{f}(r_0)+\int_{r_0}^{r}V'_{f}(s)ds\\
&\geq e^{-c_{p}}(4\pi)^{-n/2}(n-2\epsilon)\alpha(\epsilon)C(A,\epsilon)\int_{r_0}^{r}s^{n-1-2n\epsilon}e^{-s^2/4}ds
\end{aligned}
\end{eqnarray*}
Hence
\begin{equation}
{\rm Vol}_{f}(M)\geq e^{-\epsilon}(4\pi)^{-n/2}(n-2\epsilon)\alpha(\epsilon)C(A,\epsilon)\int_{r_0}^{+\infty}s^{n-1-2n\epsilon}e^{-s^2/4}ds
\end{equation}
For any fixed $A>0$ and $v>0$, the right hand side of the above inequality is a continuous function about $\epsilon$, and equals $1$ as $\epsilon\rightarrow 0$. Hence there exists a uniform constant $\epsilon_{n, A,v}$ depending only on $n, A$ and $v$, such that if $R\leq \epsilon_{n, A,v}$, we have
\begin{equation}
{\rm Vol}_f(M)\geq 1-\epsilon'_n
\end{equation}
where $\epsilon'_n$ is the constant in Theorem \ref{Yokota}. Then Theorem \ref{MainThm} follows from Theorem \ref{Yokota}.
\end{proof}

\section*{Acknowledgements}
The author is partially supported by NSFC No. 11301017, Research Fund for the Doctoral Program of Higher Education of China, the Fundamental Research Funds for the Central Universities. The author also thank the referee to point out the gap about the lower bound of the injectivity radius in the previous version and helpful suggestions.

\bibliographystyle{amsplain}

\end{document}